\documentclass[12pt]{amsart}

\usepackage{psfrag}
\usepackage{graphicx}
\usepackage{amsmath}
\usepackage{amssymb, gensymb}
\usepackage{amscd}
\usepackage{picinpar}
\usepackage{times}
\usepackage{pb-diagram}
\usepackage{graphicx}
\usepackage{xspace}
\usepackage{hyperref}
\usepackage{xcolor}
\usepackage[ruled,vlined]{algorithm2e}

\newtheorem{theorem}{Theorem}[section]
\newtheorem{lemma}[theorem]{Lemma}

\newtheorem{corollary}[theorem]{Corollary}

\begin{document}
\title{The Deformation Space of Delaunay Triangulations of the Sphere }

\author{Yanwen Luo, Tianqi Wu, Xiaoping Zhu}
%%    Address of record for the research reported here
\address{Department of Mathematics, Rutgers University, New Brunswick
 NJ, 08817}
 \email{yl1594@rutgers.edu}

\address{Department of Mathematics, Clark University, Worcester, MA, 01610}
 \email{tianwu@clarku.edu}

\address{Department of Mathematics, Rutgers University, New Brunswick
 New Jersey 08817}
 \email{xz349@rutgers.edu}

%%Current address
%%\curraddr{Department of Mathematics,
%%]}
%%\thanks will become a 1st page footnote.
\thanks{Acknowledgement: The authors were supported in part by NSF 1737876, NSF 1760471, NSF DMS FRG 1760527 and NSF DMS 1811878.}


\begin{abstract}
In this paper, we determine the topology of the spaces of convex polyhedra inscribed in the unit $2$-sphere and the spaces of strictly Delaunay geodesic triangulations of the unit $2$-sphere. These spaces can be regarded as discretized groups of diffeomorphisms of the unit $2$-sphere. Hence, it is natural to conjecture that these spaces have the same homotopy types as those of their smooth counterparts. The main result of this paper confirms this conjecture for the unit $2$-sphere. It follows from an observation on the variational principles on triangulated surfaces developed by I. Rivin. 

On the contrary, the similar conjecture does not hold in the cases of flat tori and convex polygons. We will construct simple examples of flat tori and convex polygons such that the corresponding spaces of Delaunay geodesic triangulations are not connected. 
\end{abstract}

\maketitle

\section{Introduction}

One of the fundamental problems in low dimensional topology is to identify the homotopy types of groups of diffeomorphisms of a smooth manifold. Smale \cite{Sm} proved that the group of orientation preserving diffeomorphisms of the $2$-sphere is homotopy equivalent to $SO(3)$. 

This paper studies two types of finite dimensional spaces which may be considered as discrete analogues of the group of  orientation preserving diffeomorphisms of the $2$-sphere. They are the deformation spaces of Delaunay triangulations of the unit $2$-sphere and the deformation spaces of convex polyhedra inscribed in the unit $2$-sphere. The main results of this paper show that these discrete analogues are homotopy equivalent to $SO(3)$. 

\begin{theorem}
\label{maintheorem2}
The deformation space of Delaunay triangulations of the unit $2$-sphere is homeomorphic to $SO(3)\times \mathbb{R}^k$ for some $k>0$.
\end{theorem}

\begin{theorem}
\label{maintheorem}
The deformation space of the convex polyhedra inscribed in the unit $2$-sphere whose faces are all triangles is homeomorphic to $SO(3)\times \mathbb{R}^k$ for some $k>0$.
\end{theorem}

However, we will construct explicit examples of spaces of Delaunay triangulations of convex polygons and flat tori which have different homotopy types from their smooth counterparts. Specifically, we show the spaces of Delaunay triangulations of some flat tori and spaces of Delaunay triangulations of some convex polygons are not connected.

Let $T=(V,E,F)$ denote a $2$-dimensional simplicial complex, where $V$ is the set of vertices, $E$ is the set of edges, and $F$ is the set of triangles. Any edge of in $E$ is identified with the closed interval $[0,1]$, and any triangle in $F$ is identified with a Euclidean equilateral triangle with unit length. Denote $T^{(1)}$ as the 1-skeleton of $T$, and $|T|$ as the underlying space of $T$ homeomorphic to a surface possibly with boundary. 

\subsection{Delaunay triangulations of the unit sphere}
Assume $|T|$ is homeomorphic to the unit $2$-sphere denoted by $\mathbb S^2$. An embedding $\varphi:T^{(1)}\rightarrow\mathbb S^2$ is called a \emph{geodesic triangulation} of $\mathbb S^2$ if
the restriction of $\varphi$ on each edge is a geodesic parametrized with constant speed. 
A geodesic triangulation $\varphi$ naturally divides $\mathbb S^2$ into spherical geodesic triangles. For our convenience, we will only consider the geodesic triangulations where all the spherical triangles are convex. A geodesic triangulation $\varphi$ of $\mathbb S^2$ is called a \textit{convex geodesic triangulation} if any spherical triangle in $\varphi$ is contained in some open hemisphere. Such a convex geodesic triangulation $\varphi$ is uniquely determined by the images of the vertices of $T$. 

A convex geodesic triangulation $\varphi$ is called \textit{Delaunay} if it satisfies the empty circle property, meaning that for any pair of adjacent spherical triangles $\triangle ABC$ and $\triangle ABD$, $D$ is not inside the circumcircle of $\triangle ABC$.
 This condition is equivalent to the following condition on the angles of a convex geodesic triangulation:
\begin{equation}
b+c + b' +c' - a - a'\geq 0,
\end{equation} 
where $a, b, c, a', b', c'$ are the inner angles of two neighbored triangles as in Figure \ref{edge}. 
Similarly, a convex geodesic triangulation is called \textit{strictly Delaunay} if for any pair of adjacent spherical triangles $\triangle ABC$ and $\triangle ABD$, $D$ is strictly outside the circumcircle of $\triangle ABC$.
This condition is equivalent to the following condition on the angles of a convex geodesic triangulation:
\begin{equation}
\label{delaunay}
b+c + b' +c' - a - a'> 0,
\end{equation} 

Delaunay and strictly Delaunay triangulations naturally appear in the study of discrete differential geometry and geometry processing. They are widely investigated and implemented in practice. See \cite{DSO,E} for example.
We will focus on strictly Delaunay triangulations in this paper.

	\begin{figure}[h]
		\includegraphics[width=0.4\textwidth]{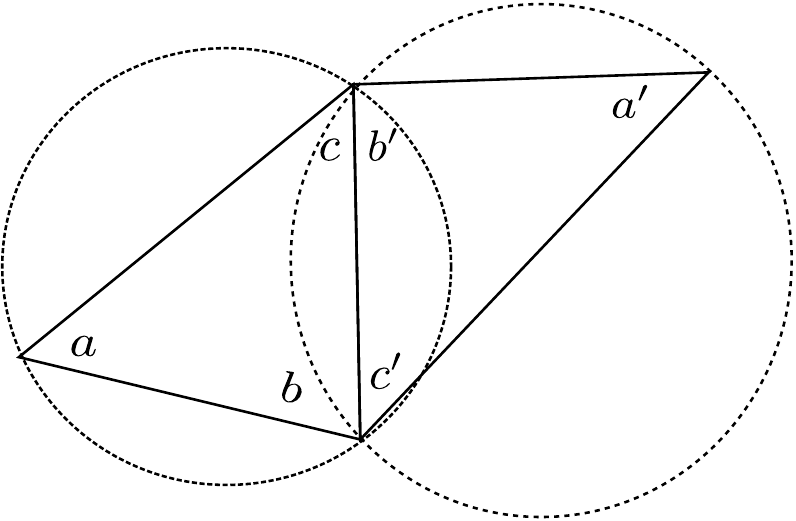}
		\caption{The edge invariant.}
			\label{edge}
	\end{figure}

Given an embedding $\psi:T^{(1)}\rightarrow\mathbb S^2$, we define \emph{the deformation space of Delaunay triangulations of the unit sphere} determined by $\psi$, denoted by $X(T,\psi)$, as the set of all strictly Delaunay convex geodesic triangulations that are isotopic to $\psi$. Then $X(T,\psi)$ is naturally a manifold of dimension $2|V|$ without boundary. Theorem \ref{maintheorem2} can be rephrased as 
\begin{theorem}
\label{main1}
Given a strictly Delaunay convex geodesic triangulation $\psi$, $X(T,\psi)$ is homeomorphic to  $\mathbb R^{2|V|-3}\times SO(3)$.
\end{theorem}

The topology of spaces of geodesic triangulations of surfaces has been studied since Cairns \cite{Ca}. These spaces are naturally discrete analogues of the diffeomorphism groups of smooth surfaces. It was conjectured that for constant curvature surfaces they are homotopy equivalent to their smooth counterparts by Connelly et al. \cite{CHHS}.
This conjecture has been confirmed by Bloch-Connelly-Henderson \cite{BCH} for convex polygons, and a new proof based on Tuttes' embedding theorem was provided by Luo \cite{Luo1}. Recently,
this conjecture was proved for the cases of flat tori and closed surfaces of negative curvature (see
Erickson-Lin \cite{EL} and Luo-Wu-Zhu \cite{LWZ1, LWZ2}). 

For the case of the unit sphere, Awartani-Henderson \cite{AH} identified the homotopy type of a subspace of the space of geodesic triangulations on the unit $2$-sphere, but the general case remains open. Theorem \ref{main1} provides an affirmative evidence about this conjecture, and we hope that it could be an intermediate step to prove the conjecture for the unit sphere.

\subsection{Convex polyhedra inscribed in the unit sphere}

Assume $|T|$ is homeomorphic to $\mathbb S^2$. An embedding $\varphi:|T|\rightarrow\mathbb R^3$ is called a \emph{polyhedral realization inscribed in the unit sphere} if $\varphi$ maps any vertex to the unit sphere and maps any face linearly to a Euclidean triangle. Such a polyhedral realization $\varphi$ is called (strictly) convex if for any triangle $\sigma\in F$, $\varphi(\sigma)$ is a face of the boundary of the convex hull of $\varphi(V)$ in $\mathbb R^3$. Given $T$, denote $Y(T)$ as the set of convex polyhedral realizations inscribed in the unit sphere. 

We say a point $q$ is \emph{inside} a convex polyhedral surface $P$ if $q$ is in the interior of the convex hull of $P$. Given a point $q$ in the unit open ball, denote $p_q:\mathbb R^3\backslash\{q\}\rightarrow\mathbb S^2$ as the radial projection centered at $q$ to the unit sphere.
We say two convex polyhedral realizations $\varphi_1,\varphi_2$ in $Y(T)$ \emph{have the same orientation} if and only if $p_{q_1}\circ\varphi_1$ is isotopic to $p_{q_2}\circ\varphi_2$ on $\mathbb{S}^2$ for $q_1$ inside $\varphi_1(|T|)$ and $q_2$ inside $\varphi_2(|T|)$. It is straightforward to check that the choice of $q_1$ and $q_2$ does not matter. 

Given a convex realization polyhedral realization $\psi$, we define \emph{the deformation space of convex polyhedra inscribed in the sphere} determined by $\psi$, denoted by $Y(T,\psi)\subset Y(T)$, as the set of all convex realizations $\varphi$ of $\mathbb{S}^2$ having the same orientation with $\psi$. Then $Y(T,\psi)$ is naturally a manifold of dimension $2|V|$ without boundary.
Theorem \ref{maintheorem} can be rephrased as 
\begin{theorem}
\label{main2}
Given a convex realization $\psi$, $Y(T,\psi)$ is homeomorphic to  $\mathbb R^{2|V|-3}\times SO(3)$.
\end{theorem}

The space of inscribed polyhedra is related to realization spaces of polytopes with a fixed combinatorial type. Steinitz \cite{St} proved that the realization space of polyhedra is a cell after standard normalization. See \cite{RG} for a detailed discussion about the realization spaces.

\subsection{Connections between the two spaces}

Denote $Y_0(T)$ as the subset of $Y(T)$ containing all the convex realizations $\varphi$ such that the origin $O=(0,0,0)$ is inside $\varphi(|T|)$. Given a convex realization $\psi$, denote $Y_0(T,\psi)=Y(T,\psi)\cap Y_0(T)$.
If $\varphi\in Y_0$, then the radial projection $p_O$ maps the triangulation structure on $\varphi(|T|)$ to a strictly Delaunay convex geometric triangulation of $\mathbb S^2$. This naturally gives a homeomorphism from $Y_0(T,\psi)$ to $X(T,p_O\circ\psi|_{T^{(1)}})$, for any convex realization $\psi$. Therefore, Theorem \ref{main1} can be reformulated as
\begin{theorem}
\label{main12}
Given a convex realization $\psi\in Y_0$, $Y_0(T,\psi)$ is homeomorphic to  $\mathbb R^{2|V|-3}\times SO(3)$.
\end{theorem}

%A convex polyhedron inscribed in the unit $2$-sphere can be regarded as an ideal convex polyhedron in the hyperbolic $3$-space, represented in the Klein model. Transforming it to the Poincare model of the hyperbolic $3$-space, the convexity of ideal polyhedron is equivalent to the empty-disk property of faces of the polyhedron: the open circumdisk of each face of the polyhedron on the $2$-sphere does not contain other vertices of the polyhedron. This implies that if the set of vertices of the polyhedron  contains no pair of antipodal points, then the projection of the polyhedron from the center determines a unique spherical Delaunay geodesic triangulations on the unit $2$-sphere. Therefore, Theorem \ref{maintheorem} relates to the problems on the homotopy types of spaces of geodesic triangulations of surfaces. 

\subsection{Organization and acknowledgement} In Section $2$, we will introduce the concept of angles structures. In Section $3$, we will determine the topology of the spaces of Delaunay triangulations of convex polygons with fixed angles. In Section $4$, we will prove Theorem \ref{main2} and Theorem \ref{main12}. In Section $5$, we will provide examples showing the homotopy types of spaces of Delaunay triangulations of flat tori and convex polygons could be not that simple.

The authors would like to thank Professor Jeff Erickson for his insightful comments.

%Notice that although three points on $\mathbb{S}^2$ don't necessarily determine  a unique geodesic triangle, they do determine a unique circle. 

%Geometrically, each circumdisk determines an open half-space in $\mathbb{H}^3$ whose boundary contains the circumdisk, and the interior does not contain any vertices of the geodesic triangulation. Then the intersection of the complements of these half-spaces forms the inscribed convex polyhedron determined by the vertices of the given triangulation. Conversely, the supporting planes of faces of an inscribed convex polyhedron intersect with $\mathbb{S}^2$, determining the corresponding open circumdisk containing no vertices. 

% We will focus on the topological property of the space of Delaunay geodesic triangulations. Denote the space of Delaunay geodesic triangulations of a surface with constant non-positive curvature as $\mathcal{X} = \mathcal{X}(S, T)$. In the case of the $2$-sphere, $\mathcal{X}(S, T)$ is the space of convex ideal polyhedra inscribed in the $2$-sphere. It is straightforward that $\mathcal{X}(S, T)$ is an open submanifold of $X(S, T)$. The main problem to discuss is that

%\begin{problem}
%Identify the homotopy type of $\mathcal{X}(S, T)$ for a surface $S$ with constant curvature.
%\end{problem}

\section{Angle structures on triangulated surfaces}
The key concept to study spaces of Delaunay triangulations is \emph{angle structure} on triangulated surfaces. This concept was proposed by Colin de Vedi\`ere \cite{CdV}, and developed by Rivin \cite{R}, Leibon \cite{Le},   Luo \cite{Luo}, Bobenko-Springborn \cite{BS}, and others. We briefly summarize the theory in the following. 

\subsection{Angle structures on triangulated surface}
Assume $|T|$ is a $2$-dimensional manifold possibly with boundary. 
A \textit{corner} in $T$ is defined as a vertex-face pair $(v,f)$ in $T$ such that the face $f$ contains $v$. It represents the inner angle of the face $f$ at the vertex $v$.
A  \textit{Euclidean angle structure} $\theta$, or an angle structure in short, on $T$ is a positive function on the set of the corners  such that 
$\theta_1 + \theta_2 +\theta_3 = \pi$ for the three angles in every face $f$.
Every angle structure can be presented as a positive vector in $\mathbb{R}^{3|F|}$. Denote $V_b\subset V$ as the set of boundary vertices, and then the \emph{edge invariant} $\alpha=\alpha(\theta)\in\mathbb R^{E\cup V_b}$ is defined as
\begin{enumerate}
	\item [(a)] 
	$\alpha_e=\theta_1+\theta_2$, if $e$ is an inner edge, and $\theta_1$ and $\theta_2$ are the two angles opposite to $e$, and

	\item [(b)] 
	$\alpha_e=\theta_1$, if $e$ is a boundary edge, and $\theta_1$ is the angle opposite to $e$, and

	\item [(c)] 
	$\alpha_v=\sum_{i}\theta_i$, if $v$ is a boundary vertex, and $\theta_i$'s are the angles at $v$.
\end{enumerate}
Denote the set of angle structures realizing a prescribed edge invariant $\bar\alpha\in\mathbb R^{E\cup V_b}$ as $\mathcal{A}(T, \bar\alpha)$. 

Given an edge length function $l\in\mathbb R^E$ satisfying the triangle inequalities, we can naturally determine a piecewise Euclidean metric on $T$ and induce an angle structure $\theta(l)$ using the inner angles in this piecewise Euclidean metric.
Notice that not every angle structure can be induced from a piecewise Euclidean metric.
%We will see that these geometric angle structures can be located by variational principles on $\mathcal{A}(T, \bar\alpha)$. 

\subsection{Variational principles of angle structures.}
Variational methods are introduced to find piecewise Euclidean surfaces with a prescribed edge invariant. The functionals in these variational principles have elegant geometric interpretations in terms of volumes of polyhedra in the hyperbolic $3$-space $\mathbb{H}^3$.

	\begin{figure}[h]
		\includegraphics[width=0.2\textwidth]{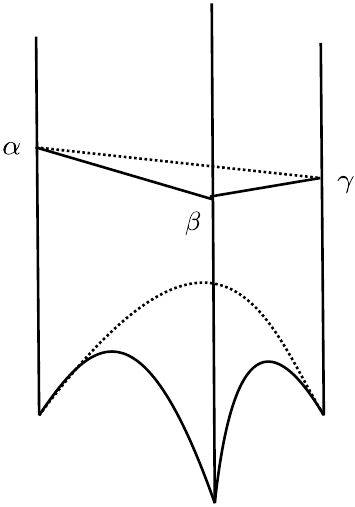}
		\caption{The volume of an ideal tetrahedron. }
			\label{energy}

	\end{figure}

For each face $f$ in $F$, an energy functional 
is defined in terms of three angles at the corners of the face in an angle structure. For a face in a Euclidean angle structure with three angles $(\alpha, \beta, \gamma)$, the energy functional is the volume of ideal hyperbolic tetrahedron whose horospherical section is similar to a Euclidean triangle with three angles $(\alpha, \beta, \gamma)$. See Figure \ref{energy}. The volume is given by 
$$V(\alpha, \beta, \gamma) = \Lambda(\alpha) + \Lambda(\beta) + \Lambda(\gamma),$$
where $\Lambda$ is the \textit{Lobachevsky function}
$$\Lambda(x) = - \int_0^x \log|2\sin \theta| d\theta.$$
The total energy for a given angle structure is defined as the sum of functionals on each face
$$\mathcal{E}(\theta) = \sum_{f_i\in F} V_i(\alpha_i, \beta_i, \gamma_i).$$
The variational principles for these energy functionals can be summarized as follows. 
\begin{theorem}{\cite{R}}
\label{angle structure}
Assume $\bar\alpha\in(0,\pi]^{E\cup V_b}$ and $\mathcal A(T,\bar\alpha)$ is nonempty, then 

\begin{enumerate}
	\item [(a)] 
the energy functional $\mathcal E$ is strictly concave down on $\mathcal{A}(T, \bar\alpha)$, and

	\item [(b)] 
there exists a unique critical point $\theta=\Theta(\bar\alpha)$ of $\mathcal E$ in $\mathcal{A}(T, \bar\alpha)$, and
	
	\item [(c)]
	$\Theta(\bar\alpha)$ is the unique angle structure in $\mathcal A(T,\bar\alpha)$ that
	could be induced from a piecewise Euclidean metric on $T$. 
\end{enumerate}

\end{theorem}

	Denote ${\mathcal A}_0(T)$ as the set of angle structures $\theta$ such that $\alpha(\theta)\in(0,\pi)^{E\cup V_b}$ and the angle sum $\sum_i\theta_i$ around any interior vertex is $2\pi$. Denote $\mathcal A_E(T)$ as the set of angle structures $\theta$ in $\mathcal A_0(T)$ that can be induced from a piecewise Euclidean metric on $T$. Notice that the angle structure induced from a Delaunay triangulation of a convex polygon in the plane belongs to $\mathcal A_E(T)$. Then by Theorem \ref{angle structure}, we have the following.
	\begin{lemma}
	\label{ae}
	If $\mathcal A_E(T)$ is nonempty, then $\mathcal A_E(T)$ is homeomorphic to $\mathbb R^k$ for some $k\geq0$.
	\end{lemma}
\begin{proof}
If $\mathcal A_E(T)$ is nonempty, then $\mathcal A_0(T)$ is nonempty. From the definition we can see that $\mathcal A_0(T)$ is an open convex subset in an affine subspace of $\mathbb R^{3|F|}$. Then its image $\alpha(\mathcal A_0(T))$ under the edge invariant map $\alpha$, which is a linear map, is an open convex subset of an affine subspace of $\mathbb R^{E\cup V_b}$. Hence,  $\alpha(\mathcal A_0(T))$ is homeomorphic to $\mathbb R^k$ for some $k\geq0$.

It remains to show that
$\bar\alpha\mapsto\Theta(\bar\alpha)$ is a homeomorphism from $\alpha(\mathcal A_0(T))$ to $\mathcal A_E(T)$. It is straightforward to show that such a map is continuous from $\alpha(\mathcal A_0(T))$ to $\mathbb R^{3|F|}$. Moreover, $\bar\alpha\mapsto\Theta(\bar\alpha)\mapsto\alpha(\Theta(\bar\alpha))$ is the identity map on $\alpha(\mathcal A_0(T))$.
By Theorem \ref{angle structure}, $\theta\mapsto\alpha(\theta)\mapsto\Theta(\alpha(\theta))$ is the identity map on $\mathcal A_E(T)$. 
Then we only need to show that the image $\Theta(\bar\alpha)$ is in $\mathcal A_E(T)$. By the definition we only need to verify that for any interior vertex $v$, the angle sum
around $v$ in $\Theta(\bar\alpha)$ is equal to the angle sum around $v$ in $\theta$. 
This is because the angle sum of an angle structure $\theta$ around an interior vertex $v$ is determined by the edge invariant $\alpha(\theta)$ as the following.
$$
\sum_{f\in F: f\ni v}\theta_{v,f}=\sum_{f\in F: f\ni v}\pi-\sum_{e\in E:e\ni v}\alpha_e.
$$
\end{proof}
\section{Delaunay Triangulations of Convex Polygons}
Assume that $|T|$ is homeomorphic to a closed disk, an embedding $\varphi:|T|\rightarrow \mathbb R^2$ is called a \emph{triangulation of a polygon} if $\varphi$ is linear on any triangle of $T$. Further such $\varphi$ is called a \emph{triangulation of a convex polygon} if the inner angle of the polygon $\varphi(|T|)$ at $\varphi(v_i)$ is less than $\pi$ for any boundary vertex $v_i$ of $T$. Such $\varphi$ is called 
\textit{strictly Delaunay} if for any pair of adjacent triangles $\triangle ABC$ and $\triangle ABD$ in $\varphi(T)$, $D$ is strictly outside the circumcircle of $\triangle ABC$.
This condition is equivalent to that 
$ a + a'< \pi$,
where $a, a'$ are the inner angles of two neighbored triangles as in Figure \ref{edge}. 

Denote $\theta(\varphi)$ as the angle structure induced from the triangulation $\varphi$, and 
$Z(T)=\{\varphi:\theta(\varphi)\in \mathcal A_E(T)\}$ as the set of strictly Delaunay triangulations of a convex polygon. We say two embeddings $\varphi,\psi$ from $|T|$ to $\mathbb R^2$ have the same orientation if $\psi\circ\varphi^{-1}$ is an orientation preserving map on $\varphi(|T|)$. Given a triangulation $\psi$ of a polygon, denote $Z(T,\psi)$ as the set of strictly Delaunay triangulations $\varphi$ of a convex polygon that have the same orientation with $\psi$. 

Furthermore, if we are given a directed edge $e_{ij}$ of $T$, denote $Z(T,\psi,e_{ij})$ as the set of strictly Delaunay triangulations $\varphi\in Z(T,\psi)$  satisfying that $\varphi(j)-\varphi(i)=(\lambda,0)$ for some $\lambda>0$. Then it is elementary to see that a triangulation in $Z(T,\psi,e_{ij})$ is uniquely determined by the induced angle structure $\theta(\varphi)$, and $\varphi(i)$, and $\varphi(j)-\varphi(i)$. Therefore, $\varphi\mapsto (\theta(\varphi),\varphi(i),\varphi(j)-\varphi(i))$ gives a homeomorphism from $Z(T,\psi,e_{ij})$ to $\mathcal A_E(T)\times\mathbb R^2\times\mathbb R_+$. 
On the other hand, the space $Z(T,\psi,e_{ij})$ is a $(2|V|-1)$-dimensional manifold if not empty, then
we have the following from Lemma \ref{ae}.
\begin{corollary}
\label{tri of polygon}
Given any Delaunay triangulation of a convex polygon $\psi$, and a directed edge $e_{ij}$, $Z(T,\psi,e_{ij})$ is homeomorphic to $\mathbb R^{2|V|-1}$.
\end{corollary}

%Next section we will reduce the spaces of Delaunay triangulations on the sphere and the spaces of convex polyhedra inscribed in the sphere to the space $Z(T,\psi,e_{ij})$. 

\section{Proof of the Main Theorems}
It is well known that the stereographic projection 
$$
\pi:(x,y,z)\mapsto\left(\frac{x}{1-z},\frac{y}{1-z}\right)
$$
gives an angle-preserving diffeomorphism from $\mathbb S^2\backslash\{(0,0,1)\}$ to $\mathbb R^2$. For a circle $\Gamma$ on $\mathbb S^2$, the stereographic projection maps $\Gamma$ to a circle on $\mathbb R^2$ if $\Gamma$ does not contain $(0,0,1)$, and maps $\Gamma\backslash\{(0,0,1)\}$ to a straight line in $\mathbb R^2$ if $\Gamma$ contains $(0,0,1)$.

We also need to recall the concept of \textit{power of a point from a circle}.  Given a point $p$ and a circle $C$ of radius $R$ centered at $q$ in the plane, the power of $p$ from $C$ is defined as 
$$power(p, C) = d^2(p,q) - R^2,$$
where $d(p,q)$ is the Euclidean distance between $p$ and $q$. It reflects the relative distance between a point and a circle. Notice that $power(p, C)<0$ if $p$ is in the interior of $C$.

Assume $|T|$ is homeomorphic to $\mathbb S^2$, and $v_0$ is a vertex of $T$, and $\psi\in Y(T)$ is a convex realization inscribed in the unit sphere.
Denote $Y(T,\psi,v_0)$ (resp. $Y(T,v_0)$, $Y_0(T,\psi,v_0)$, $Y_0(T,v_0)$) as the set of $\varphi\in Y(T,\psi)$ (resp. $\varphi\in Y(T)$, $Y_0(T,\psi)$, $Y_0(T)$) with $\varphi(v_0)=(0,0,1)$.

	\begin{figure}[h]
		\includegraphics[width=0.8\textwidth]{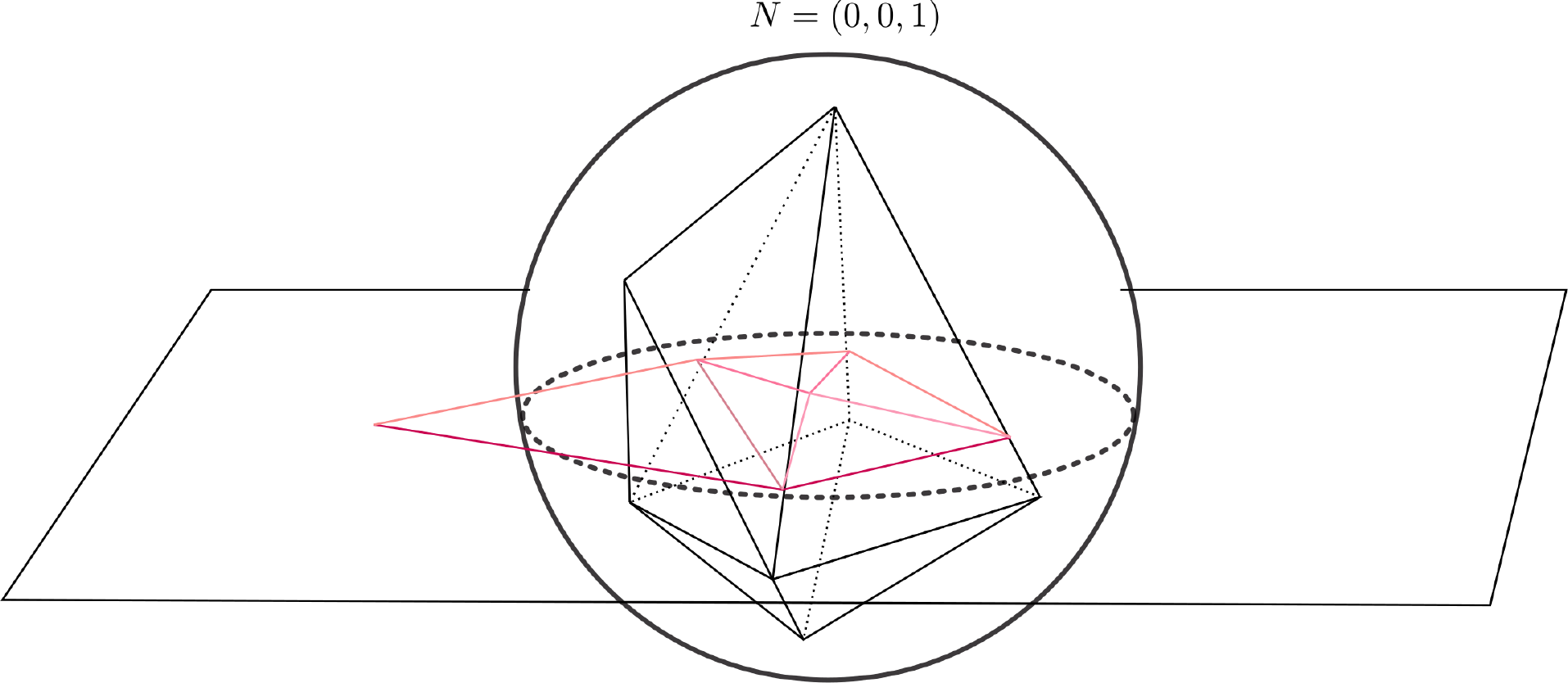}
		\caption{The stereographic projection of an inscribed convex polyhedron. }
			\label{projection}

	\end{figure}

\begin{lemma}
\label{proj}
Assume $|T|$ is homeomorphic to $\mathbb S^2$, $v_0$ is a vertex of $T$, $T_0$ denotes the subcomplex of $T$ obtained by removing the open $1$-ring neighborhood of $v_0$, and $e_{ij}$ is a directed edge in $T_0$.
\begin{enumerate}
\item[(a)] 
 There exists a map $\tilde{\pi}: Y(T, v_0) \to Z(T_0)$ induced by $\pi$ such that $\phi=\tilde \pi(\varphi)$ is the strictly Delaunay triangulation of a convex polygon determined by $\phi(v)=\pi(\varphi(v))$ for any vertex $v$ of $T_0$ (See Figure \ref{projection}).
	\item[(b)] 
 There exists a map $\tilde \eta: Z(T_0) \to  Y(T, v_0)$ induced by $\pi^{-1}$ such that $\varphi = \tilde \eta (\phi)$ is the convex realization determined by $\varphi(v)=\pi^{-1}(\phi(v))$ for any vertex $v$ of $T_0$.
	
	\item[(c)] $\tilde \pi$ and $\tilde \eta$ are inverse to each other and then $\tilde \pi$ is a homeomorphism from $Y(T,v_0)$ to $Z(T_0)$.
	
	\item[(d)] Given a convex realization $\psi\in Y(T,v_0)$, $\tilde \pi$ gives a homeomorphism from $Y(T,\psi,v_0)$ to $Z(T_0,\tilde \pi(\psi))$.
	
		\item[(e)] $\tilde \pi (Y_0(T,v_0))$ contains all the strictly Delaunay triangulations $\phi\in Z(T_0)$ satisfying that

(i) the origin $(0,0)$ is in the interior of $\phi(|T_0|)$, and

(ii) the power of $(0,0)$ with respect to the circumcircle of the triangles in $\phi(T_0)$ containing $(0,0)$ is greater than $-1$.
		
	\item[(f)] For any $\varphi\in Y(T,\psi)$, there exists a unique $\varphi_0\in Y(T,\psi,v_0)$ and $g\in SO(3)$, such that $\varphi=g\circ\varphi_0$ and $\tilde \pi(\varphi_0)\in Z(T_0,\tilde \pi(\psi),e_{ij})$. Consequently, $Y(T,\psi)$ is homeomorphic to $Z(T_0,\tilde \pi(\psi),e_{ij})\times SO(3)$.
	
	\item[(g)] For any $\varphi\in Y_0(T,\psi)$, there exists a unique $\varphi_0\in Y_0(T,\psi,v_0)$ and $g\in SO(3)$, such that $\varphi=g\circ\varphi_0$ and $\tilde \pi(\varphi_0)\in Z(T_0,\tilde \pi(\psi),e_{ij})$. Consequently, $Y_0(T,\psi)$ is homeomorphic to $(\tilde \pi (Y_0(T,v_0))\cap Z(T_0,\tilde \pi(\psi),e_{ij}))\times SO(3)$.
\end{enumerate}
\end{lemma}
\begin{proof}
(a) and (b) are true by the empty circle property of the (strict) Delaunay triangulations and the fact that the stereographic projection preserves circles. (c) is a direct consequence from the definition.

	\begin{enumerate}
	 
	\item[(d)] Given a convex realization $\varphi\in Y(T,v_0)$ and $q_1$ inside $\psi(|T|)$ and $q_2$ inside $\varphi(|T|)$, the following elementary facts related to orientations are equivalent by the definition and properties of stereographic projections. 
	\begin{enumerate}
	    \item[(i)]	$\varphi\in Y(T,\psi,v_0)$. 
	
	\item [(ii)] $\psi$ and $\varphi$ have the same orientation. 
	
	\item [(iii)] $\pi_{q_1}\circ\psi$ is isotopic to $\pi_{q_2}\circ\varphi$. 
	
	\item [(iv)] $\pi_{q_1}\circ\psi$ and $\pi_{q_2}\circ\varphi$ have the same orientation.
	
	\item [(v)] $\tilde \pi(\psi)$ and $\tilde \pi(\varphi)$ have the same orientation.
	
	\item [(iv)] $\tilde \pi(\varphi)\in Z(T_0,\tilde \pi(\psi))$.
	\end{enumerate}

	\item[(e)] We will prove that any $\tilde \pi(\varphi)$ for $\varphi\in Y_0(T,v_0)$ satisfies (i) and (ii). The other way could be proved by reversing our argument.
	If $\varphi\in Y_0(T,v_0)$, then the origin is inside $\varphi(|T|)$. Then the ray starting from the north pole passing through the origin intersects with  $\varphi(|T|)$ at a unique point $q$ in the interior of $\varphi(|T_0|)$. So  part (i) is satisfied.
	
	To prove part (ii) in (e), notice first that by the definition of $Y_0(T, v_0)$, the circumcircle $C$ of any triangle in $\varphi(T_0)$ is contained in the interior of some hemisphere bounded by a great circle $\tilde C$ in $\mathbb{S}^2$. Then $\pi(C)$ is strictly inside $\pi(\tilde C)$, then $power((0,0), \pi(\tilde C)) < power((0,0), \pi( C))  $. By the Intersecting Chords Theorem, 
	$$power((0,0), \pi(\tilde C)) = -1 < power((0,0), \pi( C)).$$ 
	
	\end{enumerate}
	(f) and (g) follow from the fact that the rotation along the $z$-axis (or the origin in the $xy$-plane) is invariant under the stereographic projection.

\end{proof}

\begin{proof}[Proof of Theorem \ref{main2}]
This is an immediate consequence of Corollary \ref{tri of polygon} and part (f) of Lemma \ref{proj}.
\end{proof}

\begin{proof}[Proof of Theorem \ref{main12}] By part (g) of Lemma \ref{proj}, we only need to show that 
$\tilde \pi (Y_0(T,v_0))\cap Z(T_0,\tilde \pi(\psi),e_{ij})$ is homeomorphic to $\mathbb R^{2|V|-3}$. It is straightforward to verify that
$$
\varphi\in \tilde \pi (Y_0(T,v_0))\cap Z(T_0,\tilde \pi(\psi),e_{ij})
$$ 
is uniquely determined by $\theta(\varphi)$, $\varphi^{-1}(0,0)$ and $d(\varphi)$, where $d(\varphi)$ is the
Euclidean diameter of $\varphi(|T_0|)$ and describes the scaling transformation needed to determine $\varphi$.
So $\mathcal F:\varphi\mapsto(\theta(\varphi),\varphi^{-1}(0,0),d(\varphi))$ gives a continuous injective map from $\tilde \pi (Y_0(T,v_0))\cap Z(T_0,\tilde \pi(\psi),e_{ij})$ to $\mathcal A_E(T_0)\times int(|T_0|)\times(0,\infty)$, where $int(|T_0|)=|T_0|\backslash \partial (|T_0|)$ is
homeomorphic to $\mathbb R^2$ and $\mathcal A_E(T_0)$ is homeomorphic to $\mathbb R^{2|V|-6}$ by Lemma \ref{ae} and a dimension counting.

It suffices to show that the image 
$Imag(\mathcal F)$ of $\mathcal F$ 
is homeomorphic to $\mathbb R^{2|V|-3}$. By part (e) of Lemma \ref{ae},
for any $(\theta,q)\in\mathcal A_E(T_0)\times int(|T_0|)$,
$$
Imag(\mathcal F)\cap\big(\{(\theta,q)\}\times(0,\infty)\big)=
(\theta,q)\times(0,d_{\theta,q})$$
where $(\theta,q)\mapsto d_{\theta,q}$ is a continuous map from $\mathcal A_E(T_0)\times int(|T_0|)$ to $(0,\infty]$. Let us assume that $\arctan(\infty)=\pi/2$ and then
$$
(\theta,q,d)\mapsto\left(\theta,q,\frac{\arctan (d)}{\arctan (d_{\theta,q})}\right)
$$
is a homeomorphism from $Imag(\mathcal F)$ to $\mathcal A_E\times int(|T_0|)\times(0,1)$, which is homeomorphic to $\mathbb R^{2|V|-6}\times\mathbb R^2\times\mathbb R=\mathbb R^{2|V|-3}$.
\end{proof}

\section{Delaunay triangulations of convex polygons and flat tori}
In this section, we will discuss the space of Delaunay geodesic triangulations of convex polygons and flat tori.

\subsection{Convex polygons}
A convex polygon $P$ in the plane is determined by the position of a sequence of cyclically ordered vertices. For a fixed convex polygon $P$ in the plane with a triangulation $\psi:T\to P$, denote the set of Delaunay triangulations of $P$ which are isotopic to $\psi$ and identical restricted to the boundary vertices with $\psi$ as $X(T, \psi)$. Notice that $X(T, \psi)$ is different from the space $Z(T, \psi)$ in Section $3$, since the positions of boundary vertices of $T$ for elements in $X$ are fixed. The following example shows that $X(T,\psi)$ may not be connected.

	\begin{figure}[h]
		\includegraphics[width=0.75\textwidth]{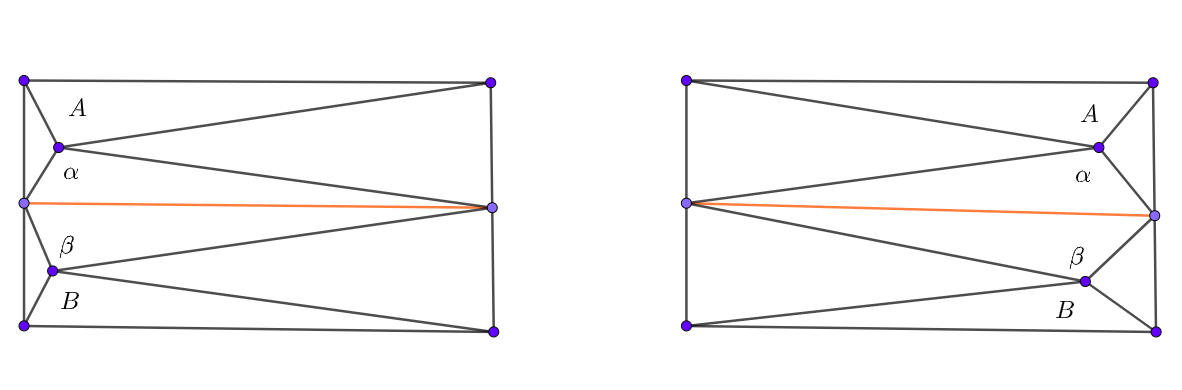}
		\caption{Counterexample: a convex polygon.}
			\label{rectangle}

	\end{figure}

In Figure \ref{rectangle}, there are nine interior edges in the triangulation, eight of which are Delaunay. The yellow edge might not be Delaunay. In Figure \ref{rectangle}, if the vertices $A$ and $B$ are close to the vertical boundaries, then $\alpha$ and $\beta$ are both acute, so we can construct two Delaunay triangulations $\tau_1$ and $\tau_2$ on the left and right. If there is a family of Delaunay triangulations connecting $\tau_1$ and $\tau_2$, the vertex $A$ or $B$ will pass the perpendicular bisector of the horizontal boundary of this rectangle. If the rectangle is flat enough, the angle sum $\alpha + \beta >\pi$ when one of $A$ and $B$ lies on the perpendicular bisector. This shows that $X(T,\psi)$ for this rectangle $P$ is not connected.

\subsection{Delaunay triangulations on flat tori.}
 Assume $|T|$ is homeomorphic to the torus $\mathbb T^2$ with a marking homeomorphism whose restriction on $T^{(1)}$ is denoted as $\psi$. An embedding $\varphi:{T}^{(1)} \to \mathbb{T}^2$ is a \textit{Delaunay geodesic triangulation} with the combinatorial type $({T},\psi)$ satisfying that
\begin{enumerate}
	\item [(a)] the restriction $\varphi_{ij}$ of $\varphi$ on each edge $e_{ij}$, identified with a unit interval $[0,1]$, is a geodesic parametrized with constant speed, and 
	\item [(b)] $\varphi$ is homotopic to $\psi$, and
	\item [(c)] equation (\ref{delaunay}) is satisfied for all edges in $T$.
\end{enumerate}
Let  $X = X(T, \psi)$ denote the set of all such geodesic triangulations, which is called the \textit{deformation space of Delaunay geodesic triangulations of $\mathbb T^2$} of combinatorial type $(T, \psi)$.

The following example shows that the space of Delaunay geodesic triangulations $X = X( T, \psi)$ may not be connected. 

	\begin{figure}[h]
		\includegraphics[width=0.5\textwidth]{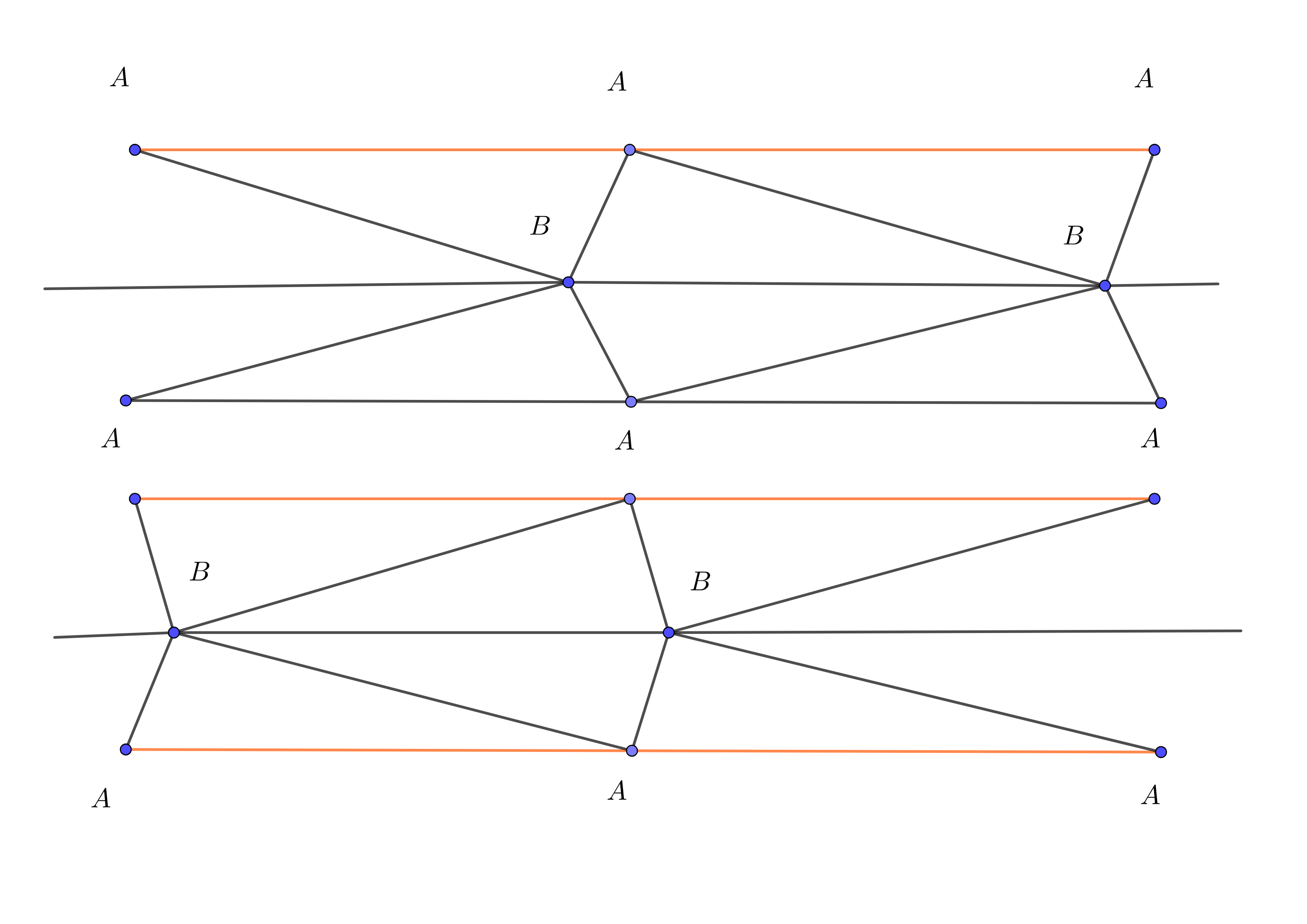}
		\caption{Counterexample: a flat torus.}
			\label{torus}
	\end{figure}

In Figure \ref{torus}  , we draw two geodesic triangulations $\tau_1$ and $\tau_2$ on a flat torus. For each geodesic triangulation, we draw two fundamental domains of this torus. The triangulation has two vertices and six edges. Fixing the vertex $A$ at a point in the universal covering, we can see that the position of the vertex $B$ determines a geodesic triangulation of this flat torus. Notice that $\tau_1$ and $\tau_2$ are both Delaunay, since all angles in these triangulations are acute when $B$ is sufficiently close to the vertical line connecting two adjacent copies of $A$ in the universal covering. 

We can choose the shape of the fundamental domain of the flat torus as shown in the picture. Then $\tau_1$ and $\tau_2$ are in two different connected components of the space of Delaunay triangulations of this flat torus. This observation is based on the following fact: any path connecting $\tau_1$ and $\tau_2$ needs to move the vertex $B$ from the right to the left. However, we can choose a flat enough fundamental domain such that when $B$ passes the perpendicular bisector of the yellow edge, the yellow edge is never Delaunay. This implies that the space $X = X(T, \psi)$ for this flat torus is not connected.

\end{document}